\documentclass[reqno]{amsart}

\usepackage{amsmath}
\usepackage{amsfonts}
\usepackage{amssymb}
\usepackage{amscd}
\usepackage{stmaryrd}
\usepackage[matrix,arrow,cmtip,rotate,curve,arc]{xy}
\SelectTips{cm}{}
\input xy
\xyoption{all}
\usepackage{graphics}
\usepackage{graphicx}
\usepackage{fancyhdr}

\DeclareSymbolFont{bbold}{U}{bbold}{m}{n}
\DeclareMathSymbol{\numberone}{\mathord}{bbold}{`1}

\newcommand{\bbZ}{\mathbb Z}

\newcommand{\mcC}{\mathcal C}

\newcommand{\mcE}{\mathcal E}

\newcommand{\mcH}{\mathcal H}
\newcommand{\mcI}{\mathcal I}

\newcommand{\mcK}{\mathcal K}

\newcommand{\mcU}{\mathcal U}

\newcommand{\Set}{\mathsf{Set}}

\newcommand{\Top}{\mathsf{Top}}

\newcommand{\Aut}{\mathop\mathsf{Aut}\nolimits}

\newcommand{\Hfib}{\mcH\textrm{-}\mathsf{fib}}

\newcommand{\map}{\mathop\mathsf{map}\nolimits}

\newcommand{\op}{\mathrm{op}}

\newcommand{\xra}[1]{\xrightarrow{#1}}

\newcommand{\ra}{\rightarrow}
\newcommand{\la}{\leftarrow}
\newcommand{\da}{\downarrow}
\newcommand{\xlra}[1]{\xrightarrow{\ #1\ }}

\newcommand{\lra}{\longrightarrow}

\newdir{c}{{}*!/-5pt/@^{(}}
\newdir{d}{{}*!/-5pt/@_{(}}
\newdir{s}{{}*!/+10pt/@{}}

\newcommand{\pbscale}{.25}
\newcommand{\pboffset}{.5}

\newcommand{\hpbsize}{15pt}
\newcommand{\hpboffset}{.5}

\newcommand{\xycorner}[3]{\save [];#2**{}?(\pbscale)="a",[];#1**{}?(\pbscale);"a"**{}?(\pboffset);"a"**\dir{-},[];#3**{}?(\pbscale);"a"**{}?(\pboffset);"a"**\dir{-} \restore}
\newcommand{\xyhcorner}[3]{\save #2="a";#1;"a"**{}?(\hpboffset);"a"**\dir{-};#3;"a"**{}?(\hpboffset);"a"**\dir{-};[]**{}?(.3)*{\scriptstyle h} \restore}

\newcommand{\corner}[3]{\xycorner{[#1]}{[#2]}{[#3]}}

\newcommand{\pb}[1][{1}]{\xycorner{[r]**{}?(#1)}{[dr]**{}?(#1)}{[d]**{}?(#1)}}

\newcommand{\hpo}{\xyhcorner{[]+<-\hpbsize,0pt>}{[]+<-\hpbsize,\hpbsize>}{[]+<0pt,\hpbsize>}}

\newcommand{\colim}{\mathop\mathrm{colim}}

\newcommand{\fib}{\mathop\mathrm{fib}\nolimits}

\newcommand{\hocolim}{\mathop\mathrm{hocolim}}
\newcommand{\hofib}{\mathop\mathrm{hofib}\nolimits}

\newcommand{\id}{\mathrm{id}}

\newcommand{\interior}{\mathop\mathrm{int}}

\newcommand{\sk}{\mathop\mathrm{sk}\nolimits}

\newlength{\hlp}

\newcommand{\rightbox}[2]{\settowidth{\hlp}{$#2$}\makebox[\hlp][r]{${#1}{#2}$}}

\entrymodifiers={+!!<0pt,\the\fontdimen22\textfont2>}

\theoremstyle{plain}
\newtheorem{theorem}{Theorem}

\newtheorem{corollary}[theorem]{Corollary}
\newtheorem{proposition}[theorem]{Proposition}

\newtheorem*{slemma}{Lemma}
\newtheorem*{scorollary}{Corollary}

\theoremstyle{definition}

\newtheorem*{sexamples}{Examples}
\newtheorem*{sdefinition}{Definition}

\newtheorem*{sassumption}{Assumption}

\theoremstyle{remark}
\newtheorem*{remark}{Remark}

\begin{document}

\title{Fibrations up to an equivalence, homotopy colimits and pullbacks}
\author{Luk\'a\v{s} Vok\v{r}\'inek}
\address{Department of Mathematics and Statistics\\Masaryk University\\Kotl\'a\v rsk\'a 2\\611 37 Brno\\Czech Republic}
\email{koren@math.muni.cz}
\thanks{Research was supported by the grant MSM 0021622409 of the Czech Ministry of Education and by the grant 201/08/0397 of the Grant Agency of the Czech Republic.}
\begin{abstract}
We gather conditions on a class $\mcH$ of continuous maps of topological spaces that allow a reasonable theory of fibrations up to an equivalence (a~map from this class) which we call $\mcH$-fibrations. The weak homotopy equivalences recover quasifibrations and homology equivalences yield homology fibrations. We study local $\mcH$-fibrations that behave nicely with respect to homotopy colimits together with universal $\mcH$-fibrations that behave nicely with respect to pullbacks. We then proceed to classify $\mcH$-fibrations up to a natural notion of equivalence.
\end{abstract}

\maketitle

Quasifibrations were invented by Dold and Thom in their study \cite{DoldThom} of the free commutative topological monoid $\mathrm{SP}(X)$ on a pointed topological space $X$. They used them to show that on connected spaces this functor turns homology groups of $X$ into homotopy groups of $\mathop{\mathrm{SP}}(X)$ by studying the effect of $\mathrm{SP}$ on cofibration sequences. One can express the main argument roughly by saying that quasifibrations behave nicely with respect to homotopy colimits. This is exactly the same principle which enabled Quillen in \cite{Quillen} to prove his Theorems A and B. Similar properties of (local) homology fibrations were exhibited by McDuff and Segal in \cite{McDuffSegal} in their proof of the group completion theorem.


In this paper we consider more generally fibrations ``up to an equivalence'', natural generalizations of quasifibrations and homology fibrations. We impose a reasonable set of conditions on the class $\mcH$ of such equivalences which guarantee nice properties of the class of $\mcH$-fibrations (i.e.~fibrations up to an equivalence from the class $\mcH$). Although we prove many useful results about them the main achievement lies perhaps in establishing this set of conditions on $\mcH$.

As in the above mentioned papers there are variants of $\mcH$-fibrations, which we call local $\mcH$-fibrations, that ensure stability under homotopy colimits but here we are more concerned with stability under pullbacks which is used in the forthcoming paper \cite{Vokrinek} about an extension of an h-principle \cite{Vassiliev} of Vassiliev. In our notation one could more precisely call the result an $\mcH$-principle, here for the class $\mcH$ of homology equivalences.

In general it is not true that a pullback of an $\mcH$-fibration is again an $\mcH$-fibration, this desirable property is already violated by the most classical case of quasifibrations. We need again a slight variation which we call a universal $\mcH$-fibration. This is a map that becomes an $\mcH$-fibration whenever pulled back along a map from a disc to the base. Our main result is that these are precisely $\mcH$-fibrations stable under pullbacks (most importantly they indeed \emph{are} $\mcH$-fibrations).

In the second part we propose a geometric condition on a map (revolving around lifting paths as in \cite{Stasheff2}) and prove that it makes the map a universal $\mcH$-fibration. In the case of universal homology fibrations we obtain a more concrete criterion as well as for universal quasifibrations.

The third part is devoted to a classification problem for $\mcH$-fibrations. This is a version of the classification theorems of Stasheff (\cite{Stasheff1}), May (\cite{May}) and others. The answer is unsatisfactory in that we fail to show that the relevant equivalence classes of $\mcH$-fibrations with a fixed ($\mcH$-equivalence class of) fibre form only a set\footnote{We believe that in general there might be a proper class of these equivalence classes.}. Assuming that this is so we prove that a classifying space exists. We give a condition (which is satisfied in most of our examples) under which we are able to identify the representing object as the classifying space of a topological monoid of self-homotopy equivalences of a certain ``localization'' of the fibre in question.

\setcounter{section}{-1}
\section{The conditions on $\mcH$ and a few examples}

Let $p:E\ra B$ be a map and $b\in B$. We denote by $\fib_bp=p^{-1}(b)$ the topological fibre and by $\hofib_bp=E\times_B\map_*(I,B)$ the homotopy fibre of $p$ over $b$.

We fix a class $\mcH\subseteq\Top$ of maps with the following properties

\begin{itemize}
\item[(0)]
$\mcH$ contains all weak equivalences.

\item[(1)]
The retract axiom: $\mcH$ is closed under retracts.

\item[(2)]
The 2-out-of-3 axiom: for composable maps $X\xra{f}Y\xra{g}Z$ the following holds: If $f\in\mcH$ then $g\in\mcH\Leftrightarrow gf\in\mcH$. If $g$ is a weak equivalence then $f\in\mcH\Leftrightarrow gf\in\mcH$.

\item[(3)]
The extension axiom: if in the diagram
\[\xymatrix@C=10pt{
E' \ar[rr] \ar[rd]_{p'} & & E \ar[ld]^p \\
& B
}\]
all the canonical maps $\hofib_bp'\ra\hofib_bp$ lie in $\mcH$ (for all choices of the basepoint) then so does the map $E'\ra E$.

\item[(4)]
The homotopy colimit axiom: suppose that $\mcI$ is a category with $B\mcI\simeq *$ and $D:\mcI\ra\mcH$ a diagram. Then all the structure maps $Di\ra\hocolim D$ also lie in $\mcH$.

\end{itemize}

We also say that $f$ is an \emph{$\mcH$-equivalence} in place of $f\in\mcH$. We now state certain strengthenings of axioms (2)-(4). These will \emph{not} be assumed in the paper unless said explicitly.

\begin{itemize}
\item[($2^+$)]
The classical 2-out-of-3 axiom: if any two of $f,g,gf$ are $\mcH$-equivalences then so is the third.

\item[($3^+$)]
If in the diagram
\[\xymatrix{
E' \ar[r]^f \ar[d]_{p'} & E \ar[d]^p \\
B' \ar[r]_g & B
}\]
the base map $g$ and all the canonical maps $\hofib_{b'}p'\ra\hofib_{f(b')}p$ are $\mcH$-equivalences (for all choices of the basepoint $b'$) then so is $f$.

\item[($4^+$)]
In addition to (4) in any homotopy pushout square
\[\xymatrix{
A \ar[r]^f \ar[d] & X \ar[d] \\
B \ar[r]_g & Y \hpo
}\]
if $f$ belongs to $\mcH$ then so does $g$. We also say that $\mcH$ is left proper.

\end{itemize}

\begin{sexamples}\hfill
\begin{itemize}
\item[a)]
Homology equivalences: properties (0)-(2) are obvious, (3) follows from the Serre spectral sequence while (4) from the homotopy colimit spectral sequence. The same works for any generalized homology theory.

\item[b)]
Weak homotopy equivalences: (0)-(3) are obvious. To prove (4) consider the diagram $D\langle 1\rangle$ of universal covers. In the diagram
\[\xymatrix{
\pi \ar[r] \ar@{=}[d] & \hocolim D\langle 1\rangle \ar[r] & \hocolim D \\
\pi \ar[r] & Di\langle 1\rangle \ar[r] \ar[u] & Di \ar[u]
}\]
the top row is a fibration sequence with $\hocolim D\langle 1\rangle$ simply connected by an easy geometric argument. Hence the middle arrow is a weak equivalence by the previous point and consequently so is the right arrow.

\item[c)]
Maps inducing isomorphism on $\pi_j$ for all $j\leq n$: only (4) needs a comment. Consider the functor $X\mapsto X[n]$ which kills all the homotopy groups $\pi_j$ for $j>n$, the $n$-th Postnikov section. The diagram $D[n]$ consists of weak homotopy equivalences and by the previous example $\hocolim D[n]\simeq Di[n]$. Also $\hocolim D[n]$ is obtained from $\hocolim D$ by attaching cells of dimension at least $n+2$ and thus $\pi_j(\hocolim D)\cong\pi_j(\hocolim D[n])\cong\pi_j(Di)$ for all $j\leq n$. This example violates ($3^+$).

\item[d)]
$(n+1)$-connected maps: again we explain why (4) holds. For $n=-1$ there is a simple argument, for $n\geq0$ reduce to the case of a diagram of connected spaces by decomposing into connected components and consider the functor $X\mapsto X\langle n\rangle$ of the $n$-connected cover. We have a fibration sequence $\Omega X[n]\ra X\langle n\rangle\ra X\ra X[n]$ of functors yielding a fibration sequence
\[\Omega D[n]\ra D\langle n\rangle\ra D.\]
By Corollary~\ref{corollary_gluing_property} applied to the class of weak equivalences there is an induced fibration sequence
\[\Omega Di[n]\ra\hocolim D\langle n\rangle\ra\hocolim D.\]
On $\pi_{n+1}$ we have a diagram
\[\xymatrix{
\pi_{n+1}(\hocolim D\langle n\rangle) \ar[r]^\cong & \pi_{n+1}(\hocolim D) \\
\pi_{n+1}(Di\langle n\rangle) \ar[r]_\cong \ar@{->>}[u] & \pi_{n+1}(Di) \ar[u]
}\]
where surjectivity follows for $n>0$ from an easy argument with the homotopy colimit spectral sequence while the case $n=0$ can be treated directly. This example violates ($2^+$).

\item[e)]
Acyclic maps, i.e.~maps inducing isomorphism in homology with local coefficients or, equivalently, maps whose homotopy fibre has singular homology of a point: the only axiom not covered in \cite{HH} is again number (4) which is just a homological (with local coefficinets) spectral sequence for the homotopy colimit.

\item[f)]
Maps whose all homotopy fibres belong to a fixed closed class $\mcC$ which is closed under extensions by fibrations: we follow \cite{Chacholski}, see also \cite{Farjoun}. Axiom (0) is part of the definition, (1) is almost trivial (a retract of a space $X$ can be written as a sequential homotopy colimit $X\ra X\ra\cdots$). Axiom (2) follows from the fibration sequence
\[\hofib_bf\ra\hofib_{g(b)}gf\ra\hofib_{g(b)}g\]
and Corollary 6.5 of \cite{Chacholski}. The same argument shows that the homotopy fibres of $\hofib_bp'\ra\hofib_bp$ are the same as those of $E'\ra E$ proving (3). The proof of (4) is more complicated and goes roughly this way: according to Theorem~9.4.~of \cite{Chacholski} it holds for the pushout category (in fact the left properness of ($4^+$) holds) and it is also not difficult to prove (4) for diagrams indexed over ordinals since the homotopy fibre of a transfinite composition is a homotopy colimit of the homotopy fibres of the partial compositions and these lie in $\mcC$ by the fibration sequence above. The general case then follows since homotopy colimits over categories with contractible nerve are generated by these two kinds of homotopy colimits.
\end{itemize}
\end{sexamples}

\section{General theory of $\mcH$-fibrations}

\begin{sdefinition}
A map $p:E\ra B$ is called an \emph{$\mcH$-fibration} if for all points $b\in B$ the canonical inclusion $\fib_bp\ra\hofib_bp$ is an $\mcH$-equivalence.
\end{sdefinition}

%
%

\begin{sdefinition}
A map $p:E\ra B$ is called a \emph{universal $\mcH$-fibration} if for all maps $\sigma:D^n\ra B$ the map $\hat{p}$ in the pullback square
\[\xymatrix{
\sigma^*E \pb \ar[r] \ar[d]_{\hat{p}} & E \ar[d]^p \\
D^n \ar[r]^\sigma & B
}\]
is an $\mcH$-fibration.
\end{sdefinition}

\begin{remark}
Equivalently, one could ask the same condition for cones on all finite simplicial complexes rather than just discs. As every such cone retracts off a disc the equivalence follows from a readily verified fact that $\mcH$-fibrations are closed under retracts. In fact we could have left out the retract axiom if we replaced all discs by cones as this is the only way in which the retract axiom is used in this paper.
\end{remark}

An easy but equally important observation is that the notion of a universal $\mcH$-fibration is closed under pullbacks. Our main result on universal $\mcH$-fibrations, Corollary~\ref{corollary_main}, asserts that they are always $\mcH$-fibrations and then clearly precisely those stable under pullbacks. This should explain the name universal.

\begin{sdefinition}
A map $p:E\ra B$ is called a \emph{local $\mcH$-fibration} if there exists a basis for topology of $B$ consisting of open subsets $U$ over which the map $p$ is an $\mcH$-fibration, i.e.~such that $p^{-1}(U)\ra U$ is an $\mcH$-fibration.
\end{sdefinition}

The local $\mcH$-fibrations over a locally contractible base are exactly those $\mcH$-fibrations stable under restrictions to open subsets as the following theorem shows. We recall from \cite{DuggerIsaksen} that a cover $\mcU$ of a topological space $B$ is called complete if the intersection $U\cap V$ of any $U,V\in\mcU$ can be expressed as a union of elements of $\mcU$.

\begin{theorem}\label{theorem_local_to_global_hurewicz}
Let $p:E\ra B$ be a continuous map
. Suppose that $B$ admits a complete cover by contractible open subsets $U$ over which $p$ is an $\mcH$-fibration (i.e.~such that $(\fib_bp\ra p^{-1}(U))\in\mcH$). Then $p$ itself is an $\mcH$-fibration.
\end{theorem}

\begin{proof}
First suppose that $B$ is contractible. Denote the cover from the statement by $\mcU$ and consider the functor $\varepsilon:\mcU\ra\Top$, $\varepsilon(U)=p^{-1}(U)$. Then \cite{DuggerIsaksen} provides
\begin{equation} \label{eqn_hocolim_decomposition}
E\simeq\hocolim\varepsilon.
\end{equation}
Observe that by the same theorem the classifying space $B\mcU$ is contractible: indeed, for the functor $\beta:\mcU\ra\Top$, $\beta(U)=U$, one gets
\[*\simeq B\simeq\hocolim\beta\simeq B\mcU\]
(as $B$ as well as every element of $\mcU$ are all contractible). As the diagram $\varepsilon$ takes place in $\mcH$ (guaranteed by the 2-out-of-3 axiom) the theorem follows from the homotopy colimit axiom.

The general case is reduced to a contractible base via the diagram
\[\xymatrix{
\hofib_bp \pb \ar[r] \ar[d] & E \ar[d] \\
\map_*(I,B) \ar[r] & B
}\]
where one needs to check that our hypotheses still hold for the left hand side map so that one can apply the previous proof. We will not need the general case and leave the details to the reader. The proof of Proposition~5 in \cite{McDuffSegal} might be useful.
\end{proof}

In particular a local $\mcH$-fibration over a CW-complex (or any locally contractible space) is an $\mcH$-fibration.

\begin{slemma}
A universal $\mcH$-fibration $p:E\ra D^n$ over a disc is a local $\mcH$-fibration.
\end{slemma}

\begin{proof}
By definition $\sigma^*(E)\ra B^n$ is an $\mcH$-fibration over any closed ball $B^n\subseteq D^n$. As open balls are increasing unions of closed ones the claim easily follows from the homotopy colimit axiom.
\end{proof}

\begin{theorem}
If $B$ has a cover by open subsets over which $p$ is a universal $\mcH$-fibration then $p$ itself is a universal $\mcH$-fibration.
\end{theorem}

\begin{proof}
Let $\sigma:|\Delta^n|\ra B$ be a map. Barycentric subdivision and the previous lemma show that $\sigma^*E\ra|\Delta^n|$ is a local $\mcH$-fibration. The proof is then completed by refering to Theorem~\ref{theorem_local_to_global_hurewicz}.
\end{proof}

\begin{theorem}
Every universal $\mcH$-fibration $p:E\ra B$ is an $\mcH$-fibration.
\end{theorem}

\begin{proof}
Fix a basepoint $b\in B$ and form the category $\mcK(B)$ whose objects are pairs $(K,\sigma)$ with $K$ a finite simplicial set and $\sigma:|CK|\ra B$ a pointed map. Morphisms are simplicial maps for which the obvious diagram commutes. Then for each such $(K,\sigma)$ we get a pullback $\mcH$-fibration $\sigma^*E\ra|CK|$, whose homotopy fibre yields a functor $F:\mcK(B)\ra\Top$. We claim now that the obvious map
\[\hocolim F\ra\hofib_bp\]
is a weak equivalence. Easily $\mcK(B)$ is filtered (as it is closed under finite colimits) and thus our claim amounts to the isomorphism $\colim\pi_*(F)\cong\pi_*(\hofib_bp)$ of homotopy groups which clearly holds as every map $S^i\ra\hofib_bp$ comes from a map $S^i\ra\sigma^*E$ and the same is true for homotopies. As the diagram lies in $\mcH$ and $\mcK(B)$ is filtered, implying contractibility of its classifying space, the composition
\[\fib_bp=F(\emptyset,b)\lra\hocolim F\lra\hofib_bp\]
then also lies in $\mcH$.
\end{proof}

\begin{corollary} \label{corollary_main}
Universal $\mcH$-fibrations are precisely those $\mcH$-fibrations that are stable under pullbacks. \qed
\end{corollary}

\begin{sdefinition}
We say that a sequence $F\xra{i}E\xra{p}B$ is an \emph{$\mcH$-fibration sequence} if there is given a null-homotopy (say to $b\in B$) of the composition $pi$ in such a way that the induced map $F\ra\hofib_bp$ is an $\mcH$-equivalence. A typical example is, for an $\mcH$-fibration $p$, the sequence $\fib_bp\hookrightarrow E\xra{p}B$ .

A commutative square is called $\mcH$-cartesian or an $\mcH$-pullback square
\[\xymatrix{
E' \ar[r] \ar[d] & E \ar[d] \\
B' \ar[r] & B
}\]
if the canonical map from $E'$ to the homotopy pullback of $B'\ra B\la E$ is an $\mcH$-equivalence.
\end{sdefinition}


\begin{theorem}
Let
\[\xymatrix{
E' \ar[r] \ar[d]_{p'} & E \ar[d]^p \\
B' \ar[r]^g & B
}\]
be a pullback square with $p$ a universal $\mcH$-fibration. Then it is an $\mcH$-pullback square.
\end{theorem}

\begin{proof}
Let $b'\in B'$ be an arbitrary point and consider the following diagram. It is constructed by replacing $p:E\ra B$ by a fibration $\hat{p}:\hat{E}\ra B$ up to a weak equivalence, considering (topological) fibres $F$ and $\hat{F}$ over $g(b')$ and then pulling everything back to $B'$. Thus $\hat{E}'$ is the homotopy pullback.
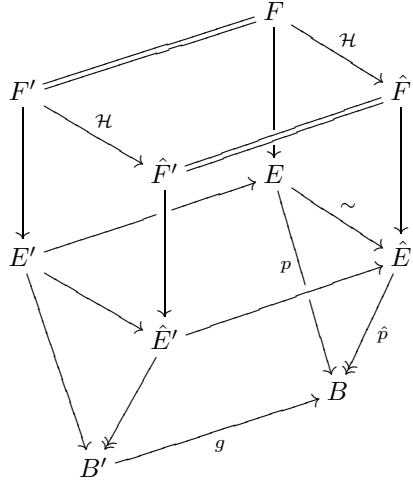
\[\xymatrix@R=5pt@C=10pt{
& & & & F \ar[rrdd]^\mcH \ar[dddd]|!{[lldddd];[rrdd]}\hole & & \\ \\
F' \ar@{=}[rrrruu] \ar[rrdd]^\mcH \ar[dddd] & & & & & & \hat{F} \ar[dddd] \\ \\
& & \hat{F}' \ar@{=}[rrrruu] \ar[dddd] & & E \ar@{-}[uuuu]|!{[ll];[rruu]}\hole \ar[rrdd]^\sim \ar[rddddd]|!{[lldddd];[rrdd]}\hole_(.4)p & & \\ \\
E' \ar[rrrruu]|!{[rruu];[rrdd]}\hole \ar[rrdd] \ar[rddddd] & & & & & & \hat{E} \ar@{->>}[lddd]^{\hat{p}} \\ \\
& & \hat{E}' \ar[rrrruu] \ar@{->>}[lddd] & & & & \\
& & & & & B & \\ \\
& B' \ar[rrrruu]_g & & & & &
}\]
By the previous corollary both vertical sequences on the left are $\mcH$-fibration sequences and the proof is finished by referring to the extension axiom for $\mcH$.
\end{proof}

For completeness we prove the gluing property for local $\mcH$-fibrations. Observe that the proof only requires $\mcH$ to be a subcategory of $\Top$ satisfying (0).

\begin{theorem} \label{theorem_hocolim_invariance}
Let $\mcI$ be a small category and $p:E\ra B$ a natural transformation between two functors $E,B:\mcI\ra\Top$ such that all the components $p_i$ are local $\mcH$-fibrations and such that for any morphism $\alpha:i\ra j$ the induced map $\fib_bp_i\ra\fib_{\alpha_*(b)}p_j$ is in $\mcH$ for all $b\in B_i$. Then the induced map
\[\hocolim E\ra\hocolim B\]
on homotopy colimits is itself a local $\mcH$-fibration.
\end{theorem}

\begin{proof}
This is really a statement about simplicial spaces. Suppose that $X$ and $Y$ are simplicial spaces, $f:X\ra Y$ is a natural transformation, and that both $X$ and $Y$ are degeneracy free on $W_n$ and $Z_n$, see \cite{GoerssJardine}. Suppose that each $W_n\ra Z_n$ is a local $\mcH$-fibration (with fibres related as in the statement). Then $\sk_n|X|\ra\sk_n|Y|$ is a map of pushouts of the rows of the diagram
\[\xymatrix{
W_n\times\Delta^n \ar[d] & W_n\times\partial\Delta^n \ar[l] \ar[r] \ar[d] & \sk_{n-1}|X| \ar[d] \\
Z_n\times\Delta^n & Z_n\times\partial\Delta^n \ar[l] \ar[r] & \sk_{n-1}|Y|
}\]
Easily the portion of the pushout over $Z_n\times(\interior\Delta^n)$ is a local $\mcH$-fibration. To get the same near any point of $\sk_{n-1}|Y|$ observe that one gets a neighbourhood basis consisting of open sets which deformation retract onto a neighbourhood in $\sk_{n-1}|Y|$ and that these deformations also work in $\sk_n|X|$. The very same argument shows that the map between full realizations is also a local $\mcH$-fibration.
\end{proof}

A natural transformation $p:E\ra B$ between two functors $E,B:\mcI\ra\Top$ is called \emph{$\mcH$-equifibred} if for each morphism $f:i\ra j$ in $\mcI$ and any $b\in Bi$ the map $\hofib_bp_i\ra\hofib_{f_*(b)}p_j$ lies in $\mcH$.

\begin{corollary} \label{corollary_gluing_property}
Let $\mcI$ be a small category and $p:E\ra B$ an $\mcH$-equifibred natural transformation between functors $E,B:\mcI\ra\Top$. Then the inclusion of the homotopy fibre $\hofib_bp_i$ of the $i$-component of $p$ into the homotopy fibre of $\hocolim E\ra\hocolim B$ over the same point $b\in Bi$ is an $\mcH$-equivalence.
\end{corollary}

\begin{proof}
Replace $p$ objectwise by a fibration $\hat{p}:\hat{E}\ra B$ over a CW-complex and apply the previous theorem and Theorem~\ref{theorem_local_to_global_hurewicz} to conclude that $\hocolim\hat{E}\ra\hocolim B$ is an $\mcH$-fibration. The topological fibre of $\hat{p}$ over $b\in Bi$ is the homotopy fibre of the $i$-component of the original transformation: $\fib_b\hat{p}=\fib_b\hat{p}_i=\hofib_bp_i$.
\end{proof}

\section{Recognizing homology fibrations}

\begin{sdefinition}
A \emph{path-lifting map} for $p:E\ra B$ is a continuous map
\[\ell:E\times_B\map(I,B)\ra E\]
where $\map(I,B)$ is a space over $B$ via the evaluation at $0$. We require of $\ell$ that it is fibrewise, i.e. $\ell_\gamma=\ell(-,\gamma)$ maps $E_{\gamma(0)}$ to $E_{\gamma(1)}$.
\end{sdefinition}

\begin{remark}
Easily a path-lifting map for $p$ induces one for any pullback.
\end{remark}

Certainly this definition is not sufficient for proving anything. Although we do not ask for associativity some form of unit axiom will be essential. Let us therefore denote by $u$ the unit map
\[u:e\mapsto\ell(e,\underline{p(e)})\]
sending $e$ to the effect to $e$ of the constant path $\underline{p(e)}$ on $p(e)$. A most straightforward requirement is that $u$ should be \emph{equal} to identity. What we recover is then one of the equivalent definitions of a Hurewicz fibration. The second most natural assumption, that $u$ should be fibrewise homotopic to identity, yields what \cite{Dold} calls a weak covering homotopy property for $p$. Both of these are easily seen to be universal quasifibrations. This is generalized by the following definition and proposition.

\begin{sdefinition}
We say that $\ell$ is an \emph{$\mcH$-path-lifting map} if the unit map $u$ is an $\mcH$-equivalence. We say that $\ell$ is a universal $\mcH$-path-lifting map if the same holds for the induced path-lifting map on any pullback along a map $D^n\ra B$. A map $p:E\ra B$ is called a \emph{strong $\mcH$-fibration} if it admits a universal $\mcH$-path-lifting map.
\end{sdefinition}

\begin{proposition}
Under $(2^+)$ every strong $\mcH$-fibration is a universal $\mcH$-fibration.
\end{proposition}

\begin{proof}
Taking a pullback along any map $D^n\ra B$ reduces to the case $B=D^n$. Any contraction of $D^n$ onto $b$ then produces a section $h:D^n\ra\map(I,D^n)$ and further a homotopy from an $\mcH$-equivalence $u$ to a map (hence $\mcH$-equivalence itself) $E\ra\fib_bp\hookrightarrow E$ via $e\mapsto\ell(e,h(p(e)))$. Since the other composition is also an $\mcH$-equivalence, ($2^+$) with (4) applied to $\fib_bp\hookrightarrow E\ra\fib_bp\hookrightarrow E\ra\cdots$ easily imply that the individual maps lie in $\mcH$.
\end{proof}

We will now prove two rather technical Theorems~\ref{theorem_recognize_homology_fibrations} and \ref{theorem_recognize_quasifibrations} that will be used in \cite{Vokrinek} for showing that a certain map is a universal homology fibration/universal quasifibration. We first make the following definition.

\begin{sdefinition}
We say that a nested sequence $E_k$ of subspaces of $E$ is a \emph{filtration} of $E$ if 
any map $K\ra E$ from a finite simplicial complex $K$ to $E$ can be homotoped into $E_k$ for some $k$.\footnote{More natural would be to require that $\hocolim E_k\ra E$ is a weak homotopy equivalence but this is not needed. Also our version is more easily made fibrewise.} We say that it is a \emph{fibrewise filtration} if this homotopy can be always chosen to be fibrewise.
\end{sdefinition}

\begin{slemma}
Consider a space $p:E\ra B$ over $B$ and a continuous fibre-preserving map $f:E\ra E$ for which there exists a fibrewise filtration $E_k$ of $E$, a family of covers $\mcU_k$ of $B$ and, for each $U\in\mcU_k$, a fibrewise homotopy
\[(E_k\cap p^{-1}(U))\times I\ra E\]
from $\id$ to $f$. Then the induced map $f_*$ in homology is surjective. The same is true for any pair $(E,E')$ with $E'=p^{-1}(B')$. When $B$ is a point $f$ is injective on all homotopy groups.
\end{slemma}

\begin{proof}
Let us represent an element of $H_i(E;\bbZ)$ by $\varphi_*(c)$ where $\varphi:K\ra E$ is a map from a finite simplicial complex $K$ into $E$, $c$ is a simplicial cycle on $K$ and we require $K$ to be a union of its $i$-simplices. We will say that $\varphi$ has base dimension $j$ if there exists a factorization
\[\xymatrix{
K \ar[r]^-\varphi \ar[d]_\rho & E \ar[d]^p \\
L \ar[r] & B
}\]
where $L$ is a $j$-dimensional simplicial complex and $\rho$ is a simplicial map which does not map any $i$-simplex into the $(j-1)$-skeleton on $L$. By our assumptions we can assume that $\varphi(K)$ lies in $E_k$ for some $k$ without changing the base dimension. We will prove that $\varphi_*(c)$ is homologous to a singular chain in the image of $f_*$ by an induction on the base dimension $j$.

For $j=0$ the image of $\varphi$ lies in a finite union of fibres and on each $E_k\cap\fib_bp$ we have a homotopy from $\id$ to $f$ giving $[\varphi_*(c)]=[f_*\varphi_*(c)]$.

Now for the induction step. By a subdivision of both $K$ and $L$ we may assume that each simplex of $K$ gets mapped to a subset $E_k\cap p^{-1}(U)$ where the homotopy is defined. For each $j$-simplex $s\in L$ this homotopy provides a fibrewise homotopy $I\times\rho^{-1}(s)\ra p^{-1}(U)$. Giving the union of these the correct orientations and gluing at $\partial I\times\rho^{-1}(s)$ where the homotopies are compatible provides a map $\hat{\varphi}$ from an ($i+1$)-dimensional simplicial complex $\hat{K}$ to $E$. The simplicial cycle on $K$ also yields a simplicial chain on $\hat{K}$ whose boundary gets mapped by $\hat{\varphi}$ to $\varphi_*(c)-f_*\varphi_*(c)+\bar{\varphi}_*(\bar{c})$ where $\bar{c}$ is a simplicial cycle on an $i$-dimensional subcomplex $\bar{K}$ of $\hat{K}$ lying over the ($j-1$)-skeleton of $L$ and $\bar{\varphi}$ is the restriction of $\hat{\varphi}$.

The relative version is the same and the case $B=*$ is straightforward.
\end{proof}

\begin{scorollary}
If the unit map of some path-lifting map for $p:E\ra D^n$ satisfies the conditions from the previous lemma then the inclusion $\iota$ of the fibre $\fib_0p$ into $E$ is a homology equivalence 
which is also injective on all homotopy groups.
\end{scorollary}

\begin{proof}
Up to homotopy the unit map $u$ factors through the fibre $\fib_0p$ over $0\in D^n$. Thinking of it as a map $(E,\fib_0p)\ra(E,\fib_0p)$ it is therefore both zero and surjective in homology by the previous lemma. This implies that the inclusion $\fib_0p\hookrightarrow E$ is a homology equivalence. The injectivity of $\iota$ on homotopy groups follows from the factorization $u|_{\fib_0p}:\fib_0p\xlra{\iota}E\lra\fib_0p$.
\end{proof}

We summarize the situation in the following definition.

\begin{sdefinition}
We say that a path-lifting map $\ell$ is \emph{sequentially homotopy unital} if there exists a fibrewise filtration $E_k$ of $E$, a family $\mcU_k$ of covers of $B$ and, for each $U\in\mcU_k$ a fibrewise homotopy
\[h_U:(E_k\cap p^{-1}(U))\times I\ra E\]
from $\id$ to the unit map $u$. We say that the unit homotopies are \emph{coherent} if any two $h_U$, $h_V$ (for $U\in\mcU_k$ and $V\in\mcU_l$) are homotopic as fibrewise homotopies from $\id$ to $u$ on the intersection of their domains.
\end{sdefinition}

\begin{theorem} \label{theorem_recognize_homology_fibrations}
A map $p:E\ra B$ admitting a path-lifting map sequentially homotopy unital is a universal homology fibration.
\end{theorem}

\begin{proof}
Both the path-lifting map and the sequential homotopy pass to all pullbacks and over discs the last corollary applies.
\end{proof}

The next theorem gives a sufficient condition for improving the conclusion of the previous theorem from a universal homology fibration to a universal quasifibration. We need to assume that for each fibre $\fib_bp$ of $p:E\ra B$ the map
\begin{equation} \label{eqn_fundamental_group}
\pi_1(\hocolim(E_k\cap\fib_bp))\xlra{\cong}\pi_1(\fib_bp)
\end{equation}
is an isomorphism. This happens in particular when $\hocolim(E_k\cap\fib_bp)\ra\fib_bp$ is a weak equivalence.

\begin{theorem} \label{theorem_recognize_quasifibrations}
Suppose that $p:E\ra B$ admits a path-lifting map coherently sequentially homotopy unital and the filtration satisfies $(\ref{eqn_fundamental_group})$. Then it is a universal quasifibration.
\end{theorem}

\begin{proof}
As usual reduce to the case $B=D^n$ by pulling back. Consider the universal cover $\tilde{E}$ of $E$, denote the composition $\tilde{E}\ra E\ra D^n$ by $\tilde{p}$ and note that $\fib_b\tilde{p}$ is determined by the following pullback square
\[\xymatrix{
\fib_b\tilde{p} \pb \ar[r] \ar[d] & \tilde{E} \ar[d] \ar[rd]^{\tilde p} \\
\fib_bp \ar[r] & E \ar[r]_p & D^n
}\]
We claim that the structure from the statement lifts to the universal cover, starting with the path-lifting map. Let $\tilde{e}\in\tilde{E}$, $\gamma\in\map(I,D^n)$ and we would like to define $\tilde{\ell}(\tilde{e},\gamma)$. We denote by $e$ the image of $\tilde{e}$ in $E$ and assume that $e\in E_k$. To decide which of the preimages of $\ell(e,\gamma)$ to take for $\tilde{\ell}(\tilde{e},\gamma)$ we consider a concatenation of paths $h_U(e,t)$ and $\ell(e,\gamma|_{[0,t]})$ which is a path that starts at $e$ and finishes at $\ell(e,\gamma)$. By our assumptions although not well-defined its homotopy class rel $\partial I$ is. Therefore the evolution on $\tilde{E}$ starting at $\tilde{e}$ finishes at a well-defined point which will be our $\tilde{\ell}(\tilde{e},\gamma)$. The effect of $\tilde{\ell}$ on points $\tilde{e}$ with image $e$ not in $E_k$ is defined by choosing a path $\wp$ inside one fibre from $e$ to $e'\in E_k$ and applying the previous to $\wp*h_U(e',-)*(u\circ\wp^{-1})$ instead of $h_U(e,-)$. It is easy to check that with this definition $\tilde\ell$ is a well-defined path-lifting map which again admits a sequential homotopy to identity.

Therefore we conclude that $\fib_b\tilde{p}\ra\tilde{E}$ is a homology equivalence or in other words $\fib_bp\ra E$ induces an isomorphism in homology with any local coefficients. As it is also injective on $\pi_1$ 
by the last corollary it must be a weak homotopy equivalence by Proposition~1.4 of \cite{HH}.
\end{proof}

\section{Classifying $\mcH$-fibrations}

Fibre bundles with fibre $F$ are classified by homotopy classes of maps into $B\Aut(F)$ where $\Aut(F)$ denotes the topological group of homeomorphisms (or diffeomorphisms in the fibrewise smooth case) and a similar situation occurs (see \cite{May}) for fibrations where $\Aut(F)$ is now to mean the topological monoid of self homotopy equivalences of $F$. In this section we will try to solve this question for $\mcH$-fibrations. This has two parts: showing that the classifying object exists (under certain conditions; in particular we have to ensure that the classified objects - $\mcH$-fibrations up to certain equivalence relation - form only a set) and identifying it.

In addition to (0)-(4) we will be assuming ($4^+$).

\begin{sdefinition}
We say that a morphism
\[\xymatrix@C=10pt@R=15pt{
E_0 \ar[rr] \ar[dr]_{p_0} & & E_1 \ar[dl]^{p_1} \\
& B
}\]
of $\mcH$-fibrations over the same base $B$ is a \emph{fibre $\mcH$-equivalence} if for each $b\in B$ the induced map $\fib_bp_0\ra\fib_bp_1$ is an $\mcH$-equivalence. We say that two $\mcH$-fibrations are \emph{fibre $\mcH$-equivalent} if there exists a zig-zag of fibre $\mcH$-equivalences connecting them.
\end{sdefinition}

\begin{sdefinition}
We say that universal $\mcH$-fibrations $p_0:E_0\ra B$ and $p_1:E_1\ra B$ over the same base $B$ are \emph{concordant} if there exists a universal $\mcH$-fibration $p:E\ra B\times I$ such that $p_i$ is isomorphic to $p|_{B\times\{i\}}$.
\end{sdefinition}

\begin{remark}
A fibre $\mcH$-equivalence is automatically an $\mcH$-equivalence by the extension axiom but converse is generally not true. The advantage of our choice of an equivalence is that it passes to all pullbacks.
\end{remark}

\begin{slemma}
Universal $\mcH$-fibrations $p_0,p_1$ over the same base $B$ are concordant if and only if they are fibre $\mcH$-equivalent.
\end{slemma}

\begin{proof}
In one direction a concordance gives fibre $\mcH$-equivalences at the top row 
\[\xymatrix@R=20pt{
E_0 \corner{r}{dr}{dd} \ar[r] \ar[dd]_{p_0} & E \ar[d]^p & E_1 \corner{l}{dl}{dd} \ar[l] \ar[dd]^{p_1} \\
& B\times I \ar[d]|*+<2pt>{\scriptstyle pr} \\
B \ar@{c->}[ur]^{i_0} \ar@{=}[r] & B & B \ar@{=}[l] \ar@{d->}[ul]_{i_1}
}\]
connecting $p_0$ with $p_1$ through a universal $\mcH$-fibration $E\ra B$. In the opposite direction if $f:E_0\ra E_1$ is a fibre $\mcH$-equivalence then consider
\[p:E=(E_0\times[0,1/2])\cup(M_f\times[1/2,1))\cup(E_1\times\{1\})\lra B\times I\]
where we denote $M_f=(E_0\times I)\cup_f E_1$ the mapping cylinder of $f$. The claim is that $p$ is a universal $\mcH$-fibration and hence a concordance between $p_0$ and $p_1$. Taking any pullback $\hat{E}$ along a map from $D^n$ one observes that it sits in the following homotopy pushout square
\[\xymatrix{
& \hat{E}_0|_A \ar[r]^\mcH \ar[d] & \hat{E}_1|_A \ar[d] \\
\fib_0\hat{p}_0 \ar[r]_\mcH & \hat{E}_0 \ar[r]_\mcH & \hat{E} \hpo
}\]
with $A\subseteq D^n$ some closed subspace (namely the preimage of $B\times[1/2,1]$). Both maps at the bottom are $\mcH$-equivalences, the second one by ($4^+$). As the fibre $\fib_0\hat{p}$ is homotopy euivalent to a fibre of either $E_0$ or $E_1$ the map $\fib_0\hat{p}_0\ra\fib_0\hat{p}$ is also an $\mcH$-equivalence. By 2-out-of-3 property $\hat{E}$ is an $\mcH$-fibration.
\end{proof}

\begin{sassumption}
We assume that for every choice of a base $B$ and a fibre $F$ there exists \emph{only a set} of fibre $\mcH$-equivalence classes of universal $\mcH$-fibrations over $B$ with fibre $\mcH$-equivalent to $F$.
\end{sassumption}

Consider the functor $\Hfib_F:\Top_*^\op\ra\Set_*$ sending a pointed topological space $B$ to the set of fibre $\mcH$-equivalence classes of universal $\mcH$-fibrations $p:E\ra B$ over $B$ equipped with an $\mcH$-equivalence $F\ra\fib_*p$ which we require to be preserved by the fibre $\mcH$-equivalences. This is the first main theorem of this section.

\begin{theorem}
Under the above assumption the functor $\Hfib_F$ is representable on the homotopy category of pointed (spaces homotopy equivalent to) CW-complexes. 
\end{theorem}

\begin{proof}
We will verify the assumptions of the Brown representability theorem. By definition pullbacks of a fixed universal $\mcH$-fibration along homotopic maps are concordant and hence fibre $\mcH$-equivalent. In particular any universal $\mcH$-fibration over $B\times I$ is fibre $\mcH$-equivalent to a pullback along the projection $\pi:B\times I\ra B$. Therefore
\[\pi^*:\Hfib_F(B)\ra\Hfib_F(B\times I)\]
is a bijection. This formally implies homotopy invariance. As every CW-complex is well-pointed we may assume that all universal $\mcH$-fibrations are trivialized \emph{near} the basepoint. This easily implies the product axiom. It remains to verify, for a map $f:A\ra B$, exactness of
\[\Hfib_F(B\cup_fCA)\ra\Hfib_F(B)\ra\Hfib_F(A).\]
Therefore let $p:E\ra B$ be a universal $\mcH$-fibration which is concordant to $A\times F\ra A$ over $A$. Gluing this concordance to $p$ provides a universal $\mcH$-fibration over $B\cup_f(A\times I)$ which is trivial over (a neighbourhood of) $A\times\{1\}$ and therefore passes to the mapping cone.

Since the representing object is easily connected by $\Hfib_F(S^0)=*$ the representability extends to all CW-complexes, see \cite{Brown}.
\end{proof}

\begin{sdefinition}
We denote the representing object by $B_\mcH^F$.
\end{sdefinition}

We would like to identify $B_\mcH^F$ as the classifying space of 
a certain monoid of self-maps which represents the connected component $\mcH_F$ of $\mcH$ containing $F$. We make the following assumption which is often met.

\begin{sassumption}
Instead of the previous assumption we require now (a stronger condition as we will see) that there exists a small category $\hat\mcH_F$ together with a homotopy terminal functor $\hat{\mcH}_F\ra\mcH_F$ so that all the homotopy colimits over $\mcH_F$ exist and can be computed by passing to $\hat{\mcH}_F$.\footnote{In fact $B_\mcH^F$ should be the classifying space $B\mcH_F$ of $\mcH_F$ itself and the total space of the universal universal $\mcH$-fibration $\hocolim\limits_{F'\in\mcH_F}F'$. It seems that this is indeed the case whenever the classifying space and the homotopy colimit make sense (and this explains our assumptions on the existence of $\hat\mcH_F$ which is to replace the large category $\mcH_F$) but we will not prove it in this paper. This is also related to our assumption that $\Hfib_F$ should form a set.} To make things easier we pretend that $\mcH_F$ is small itself.
\end{sassumption}

Under this smallness assumption we are able to construct localizations (here $\mcH$ is thought of as the class of local equivalences): let $F_\mcH$ be the homotopy colimit $\hocolim\limits_{F\ra F'}F'$ with $F\ra F'$ running over the comma category $F\da\mcH$ (or rather $F\da\hat\mcH$). We denote the component of the universal cone corresponding to $\id_F$ by $\ell_F$. By (4) it is an $\mcH$-equivalence. Since every $\mcH$-equivalence $f:F\ra F'$ induces a functor $f^*:F'\da\mcH\ra F\da\mcH$ we obtain a map on the homotopy colimits as in the diagram
\[\xymatrix{
F_\mcH & F'_\mcH \ar[l]_{f^\mcH} \\
F \ar[r]_f \ar[u]^{\ell_F} & F' \ar[lu] \ar[u]_{\ell_{F'}}
}\]
with the lower triangle homotopy commutative (being part of the universal cone to the hocolim) and the upper strictly commutative. The top map is a weak equivalence since $f^*$ is homotopy terminal: the comma category $(F\ra F'')\da f^*$ is equivalent to $\tilde F\da\mcH$ for $\tilde F$ the pushout of $F'\la F\ra F''$. A space $F$ is called $\mcH$-local if the localization map $\ell_F$ is a weak equivalence (or equivalently if all the $\mcH$-equivalences from $F$ admit a right homotopy inverse). Easily any $\mcH$-equivalence between $\mcH$-local spaces is a weak equivalence. Applying the construction above to the localization map $\ell_F:F\ra F_\mcH$ we get
\[\xymatrix{
F_\mcH & (F_\mcH)_\mcH \ar[l]_\sim \\
F \ar[r]_{\ell_F} \ar[u]^{\ell_F} & F_\mcH \ar[lu] \ar[u]_{\ell_{F_\mcH}}
}\]
The diagonal map is homotopic to $\id_{F_\mcH}$ by a homotopy $F_\mcH\ra\map(I,F_\mcH)$ that is constructed from the universal property of the homotopy colimit
. Hence $\ell_{F_\mcH}$ is a weak equivalence and therefore $F_\mcH$ is $\mcH$-local.

\begin{slemma}
Let $B$ be a topological space of a homotopy type of a CW-complex. Then every class in $\Hfib_F(B)$ is represented by a fibration $p:E\ra B$ with an $\mcH$-local fibre. Two such fibrations lie in the same class if and only if they are fibre weak homotopy equivalent.
\end{slemma}

\begin{proof}
Let $p:E\ra B$ be a universal $\mcH$-fibration with homotopy fibre $F=\hofib_*p$. Then one can decompose $E\simeq\hocolim\limits_{\sigma\in\Delta SB}F_\sigma$ where $\Delta SB$ is the simplex category of the singular simplicial set $SB$ associated to $B$ and $F_\sigma=p^{-1}(\sigma)$. Since $p$ is a universal $\mcH$-fibration the diagram $F_\sigma$ takes place in $\mcH_F$. Applying the localization as above one gets $E^{\mcH}=\hocolim\limits_{\sigma\in(\Delta SB)^\op}(F_\sigma)_\mcH$, a space over $B(\Delta SB)^\op$.
\[\xymatrix{
\rightbox{E^\mcH={}}{\hocolim\limits_{\sigma\in(\Delta SB)^\op}(F_\sigma)_\mcH} \ar[d] & \hocolim\limits_{\sigma\in\Delta SB}F_\sigma \ar[l]_-\varphi \ar[r]^-\sim \ar[d] & E \ar[d] \\
B(\Delta SB)^\op & B\Delta SB \ar[l]^-\sim \ar[r]_-\sim & B
}\]
The map $\varphi$ is induced on the homotopy colimits by the localization maps $\ell_{F_\sigma}$ and is a fibre $\mcH$-equivalence between local $\mcH$-fibrations. The base map is, under the identification $B(\Delta SB)^\op\cong B\Delta SB$, homotopic\footnote{This absorbs the non-strict naturality of $\ell$.} to identity. Replacing all the vertical maps by fibrations we see that over $B\Delta SB$ our $E$ is equivalent to a fibration with $\mcH$-local fibre as the diagram $(F_\sigma)_\mcH$ consists of weak equivalences and so by Corollary~\ref{corollary_gluing_property} the fibre of $E^\mcH\ra B(\Delta SB)^\op$ is weakly equivalent to $F_\mcH$. Since $B$ has the homotopy type of a CW-complex, $B\Delta SB\ra B$ is a (strong) homotopy equivalence and thus the same holds over $B$.
\end{proof}

We can now prove the second main theorem of this section.

\begin{theorem}
$B_\mcH^F\simeq B\mcE F_\mcH$, the classifying space of the topological monoid of self-homotopy equivalences of (the CW-approximation of) the localization $F_\mcH$. In particular in this case $\Hfib_F$ forms only a set.
\end{theorem}

\begin{proof}
This follows from the last lemma since May proves in \cite{May} that $B\mcE F_\mcH$ classifies fibre weak homotopy equivalence classes of fibrations with fibre of the weak homotopy type of $F_\mcH$.
\end{proof}

\end{document}